\documentclass[a4paper, 12pt]{article}
\usepackage{fullpage}
\usepackage{amsmath, amsthm, amssymb, mathrsfs, graphicx}
\usepackage{ifthen, subfigure}
\usepackage[usenames]{color}
\usepackage{url}

\usepackage{natbib}
\bibpunct{(}{)}{;}{a}{}{,}

\theoremstyle{plain}
\newtheorem{thm}{\indent Theorem}

\newtheorem{prop}{\indent Proposition}

\theoremstyle{remark}
\newtheorem{remark}{\indent Remark}

\theoremstyle{definition}
\newtheorem{defn}{\indent Definition}

\theoremstyle{definition}

\theoremstyle{definition}
\newtheorem*{PRalg}{\indent PR Algorithm}

\newcommand{\A}{\mathscr{A}}
\newcommand{\B}{\mathscr{B}}
\renewcommand{\AA}{\mathbb{A}}

\newcommand{\Y}{\mathscr{Y}}

\newcommand{\YY}{\mathbb{Y}}

\newcommand{\FF}{\mathbb{F}}

\newcommand{\RR}{\mathbb{R}}         

\newcommand{\E}{\mathsf{E}}

\newcommand{\prob}{\mathsf{P}}

\newcommand{\eps}{\varepsilon}
\renewcommand{\phi}{\varphi}

\newcommand{\iid}{\overset{\text{\tiny iid}}{\,\sim\,}}

\begin{document}

\title{Asymptotically optimal nonparametric empirical Bayes via predictive recursion}
\author{{Ryan Martin} \\ Department of Mathematics, Statistics, and Computer Science \\ University of Illinois at Chicago \\ \url{rgmartin@math.uic.edu} }
\date{\today}

\maketitle

\begin{abstract}
An empirical Bayes problem has an unknown prior to be estimated from data.  The predictive recursion (PR) algorithm provides fast nonparametric estimation of mixing distributions and is ideally suited for empirical Bayes applications.  This paper presents a general notion of empirical Bayes asymptotic optimality, and it is shown that PR-based procedures satisfy this property under certain conditions.  As an application, the problem of in-season prediction of baseball batting averages is considered.  There the PR-based empirical Bayes rule performs well in terms of prediction error and ability to capture the distribution of the latent features.    

\medskip

\end{abstract}

\section{Introduction}
\label{S:intro}

In large-scale inference problems, the work of Stein suggests that component-wise optimal procedures are typically sub-optimal in the simultaneous inference problem.  The common theme in all works related to simultaneous inference is a notion of ``borrowing strength''---using information about all cases for each component problem.  An important example is the false discovery rate controlling procedure of \citet{benjamini.hochberg.1995} which uses the data itself to determine the critical region for the sequence of tests.  Shrinkage rules, penalized estimation, and hierarchical Bayes inference all can be given a similar ``information sharing'' interpretation.  

One interesting approach to simultaneous inference is \emph{empirical Bayes}, where a fully Bayesian model is assumed but, rather than elicitation of subjective priors or construction of non-informative objective priors, one uses the data itself to estimate the prior.  Parametric empirical Bayes, where a parametric form is assumed for the unknown prior, has been given considerable attention in the literature; see \citet{efron2010book} and the references therein.  When the number of cases is relatively small, the parametric approach is most reasonable.  Indeed, for Robbins' brand of nonparametric empirical Bayes to be successful, a tremendously large number of cases are needed.  But high-dimensional inference problems are now commonplace in statistical applications, so nonparametric empirical Bayes is now a promising area of research.  \citet[][p.~369]{efron2003} writes
\begin{quote}
What was unimaginable [then] is commonplace today.  Nonparametric empirical Bayes applies in an almost off-the-shelf manner to microarrays.  
\end{quote}

Theoretical analysis of empirical Bayes procedures looks at the limiting properties of the corresponding risk.  After a description of the decision problem and empirical Bayes approach in Sections~\ref{S:decision.theory}--\ref{S:empirical.bayes}, I propose an apparently new notion of asymptotic optimality.   Here I say that an empirical Bayes rule is asymptotically optimal if its risk (a function of observable data) converges almost surely to the Bayes risk.  Compare this to the classical definition of asymptotic optimality in \citet{robbins1964} based on convergence in mean of the empirical Bayes risk.  While neither definition is mathematically stronger than the other, I believe there is a considerable difference from a statistical point of view.  In particular, convergence in mean is not especially meaningful to a Bayesian who does not believe in averaging risk over the sample space.  Theorem~\ref{thm:optimality} gives a set of sufficient conditions for asymptotic optimality in this apparently new almost-sure sense.  

To implement nonparametric empirical Bayes, one needs a nonparametric estimate of the prior/mixing distribution.  This, in itself, is a challenging theoretical and computational problem.  The most popular techniques are based on nonparametric maximum likelihood and kernel estimators.  Two recent references on these in the context of empirical Bayes inference are \citet{brown.greenshtein.2009} and \citet{jiangzhang2009}.  But these methods can be computationally expensive and they are primarily focused on the Gaussian location problem.  A promising alternative is the \emph{predictive recursion} (PR) algorithm, designed for fast nonparametric estimation of mixing distributions in arbitrary mixture model problems, not only Gaussian; see \citet{nqz} and \citet{newton02}.  PR seems ideally suited for the empirical Bayes problem for, given the PR estimate, a plug-in empirical Bayes estimate of the optimal Bayes rule is immediately available.  

Performance of the PR-based empirical Bayes rule depends on convergence properties of the estimates produced by PR, and a fairly detailed picture of PR's convergence properties is now available.  For finite mixtures, \citet{ghoshtokdar} proved convergence of PR under strong conditions on the mixture kernel; \citet{martinghosh} extend this result using tools from stochastic approximation theory; and \citet{pr-finite} established a nearly root-$n$ rate of convergence.  The general case, described in more detail in Section~\ref{S:estimation}, was first attacked by \citet{tmg}.  They showed that, under suitable conditions, the PR estimates of the mixing and mixture distributions are both strongly consistent in the weak- and $L_1$-topologies, respectively.  Later, \citet{mt-rate} established convergence properties of the PR estimates under model mis-specification, and also gave a bound on the rate of convergence. 

In Section~\ref{S:estimation}, I use the known convergence theory for PR together with Theorem~\ref{thm:optimality} to show that the PR-based empirical Bayes rules are asymptotically optimal, under certain conditions, in hypothesis testing and point estimation problems.  Section~\ref{S:examples} contains a comparison of the PR-based empirical Bayes rules with several other parametric and nonparametric empirical Bayes rules in an interesting example of predicting batting averages in major league baseball.  It turns out that the PR-based rule is competitive with the others in the prediction problem, but is more flexible and gives a realistic picture of the distribution of latent hitting abilities.   \citet{mt-test} make a similar conclusion concerning the potential of PR-based empirical Bayes in the large-scale multiple testing applications.  These results together suggest that PR-based empirical Bayes is a promising alternative to existing methods and worthy of further investigation.

\section{The decision problem}
\label{S:decision.theory}


\subsection{Basic definitions}
\label{SS:defs}

The general decision problem has several components.  First is parameter space $\Theta$ that contains the unknown quantity of interest $\theta$, often called the ``state of nature.'' Second is an action space $\AA$, containing all possible actions, or decisions, $a$.  Third, there is a loss function $L(a,\theta) \geq 0$ that represents the penalty for taking action $a$ when the state of nature is $\theta$.  Finally, there is observable data $Y$ taking values in a measurable space $(\YY,\Y)$, equipped with a $\sigma$-finite measure $\mu$.  When the state of nature is $\theta$, the sampling distribution of $Y$, taking values in $\YY$, is $\prob_\theta$ and its density is $p_\theta = d\prob_\theta/d\mu$.  In the theoretical analysis that follows, I shall take each of these components as given.  However, these components themselves---particularly the loss function $L(a,\theta)$ and the model $p_\theta$---are often quite difficult to elicit in practice.  For this reason, there has been extensive work on loss and model robustness \citep[e.g.,][Sec.~3.10--3.11]{ghosh-etal-book}.  

With these four components in place, I can now describe the statistical decision problem.  When data $Y=y$ is observed, action $\delta(y) \in \AA$ is taken.  Action $\delta(y)$ is called a decision rule.  Then the average loss, or \emph{risk}, of decision rule $\delta$ when $\theta \in \Theta$ is the true state of nature is defined as 
\[ R(\delta,\theta) = \int L(\delta(y),\theta) p_\theta(y) \,d\mu(y). \]
For each decision rule $\delta$ there is a risk function $R(\delta,\cdot)$, and the goal of non-Bayesian decision theory is to choose the decision rule $\delta$ whose risk function $R(\delta,\cdot)$ is the ``smallest'' in some sense.  Often there is no such rule $\delta$ which gives a uniformly smallest risk function; in such cases, concessions must be made by imposing certain constraints, like unbiasedness or equivariance \citep{lehmann.casella.1998}.

\subsection{Bayesian decision theory}
\label{SS:bayes}

In the Bayesian decision problem, there is an additional piece of input required---a prior distribution for $\theta$.  Equip $\Theta$ with an appropriate $\sigma$-algebra $\B$ and let $F$ be a probability measure defined there.  On the product space $(\YY \times \Theta, \Y \otimes \B)$, define a probability measure by the density $p_\theta(y) \,dF(\theta) \,d\mu(y)$.  Two quantities related to this joint distribution are the marginal for $Y$, namely, 
\[ p_F(y) = \int_\Theta p_\theta(y) \,dF(\theta), \]
and the conditional distribution of $\theta$ given $Y=y$, described by Bayes' formula, 
\[ dF(\theta \mid y) = \{p_\theta(y) / p_F(y)\} \,dF(\theta). \]

When the prior $F$ is known, there is a well-developed Bayesian decision theory, described next.  On the other hand, when $F$ is unknown, as is often the case in practice, some special considerations are needed; see Section~\ref{S:empirical.bayes}.  When $F$ is known, define the Bayes risk of a decision rule $\delta$ to be the average risk $R(\delta,\theta)$ as $\theta$ various according to the prior $F$; in symbols, 
\[ \rho(\delta,F) = \int_\Theta R(\delta,\theta) \,dF(\theta). \]
The Bayesian decision-theorist seeks the decision rule $\delta=\delta_F$ that minimizes the Bayes risk $\rho(\delta,F)$.  I will write $\rho(F)=\rho(\delta_F,F)$ for this minimal Bayes risk.  Below I discuss the two most common decision problems: hypothesis testing and point estimation.  

The general hypothesis testing problem considers $H_0: \theta \in \Theta_0$ versus $H_1: \theta \notin \Theta_0$, where $\Theta_0 \subset \Theta$ has positive prior probability, i.e., $F(\Theta_0) > 0$.  Here the action space is $\AA = \{a_0,a_1\}$ where $a_i = \text{``choose hypothesis $i$''}$.  A Type~I error is choosing $a_1$ when $H_0$ is true, and a Type~II error is choosing $a_0$ when $H_1$ is true.  A typical loss function in such testing problem is given by $L(a_1,\theta) = \kappa_1 I_{\Theta_0}(\theta)$ and $L(a_0,\theta) = \kappa_2 (1-I_{\Theta_0}(\theta))$, where $\kappa_1,\kappa_2$ are finite positive numbers representing the cost of a Type~I, Type~II error, respectively.  The corresponding risk function is then a linear combination of the Type~I and Type~II error probabilities.  The Bayes rule is given by
\[ \delta_F(y) = \begin{cases} a_0 & \text{if $F(\Theta_0 \mid y) > r$} \\ a_1 & \text{if $F(\Theta_0 \mid y) \leq r$} \end{cases} \]
where $F(\Theta_0 \mid y)$ is the posterior probability for $\Theta_0$, given $Y=y$, and $r = \kappa_2/(\kappa_1 + \kappa_2)$ is the relative cost of a Type~II error.  These details are given in \citet[][pp.~163--164]{berger1984}.  It is interesting that, for a point-null $H_0: \theta = \theta_0$, the quantity $F(\{\theta_0\} \mid y)$ is exactly the local false discovery rate that has appeared fairly recently in the large-scale multiple testing context \citep[e.g.,][]{mt-test, suncai2007, efron2010book}.  

For the estimation problem, I shall assume $\theta$ is the estimand, so that $\AA = \Theta$.  The most common loss function in such problems is square-error loss, i.e., $L(a,\theta) = \|a-\theta\|^2$, but other losses can be handled similarly.  For square-error loss, the Bayes rule $\delta_F(y)$ is the posterior mean of $\theta$ given $Y=y$, i.e., $\delta_F(y) = \int_\Theta \theta \,dF(\theta \mid y)$.

\section{Empirical Bayes}
\label{S:empirical.bayes}


\subsection{Setup, motivation, and classical developments}
\label{SS:setup}

In the previous section, there was a single observation $Y$ (not necessarily real-valued) and a corresponding single parameter $\theta$ (also not necessarily real-valued).  Corresponding hierarchical model for $Y$ is as follows:
\begin{equation}
\label{eq:hierarchy}
Y \mid \theta \sim p_\theta(y) \quad \text{and} \quad \theta \sim F,
\end{equation}
In this case, very little can be done when $F$ is unknown; indeed, $Y$ provides information about just a single observation from $F$ which, in turn, contributes nothing to one's lack of knowledge about $F$.  However, nowadays, there are applications which can be modeled by many samples from the hierarchical model \eqref{eq:hierarchy}.  Specifically, pairs $(Y_1,\theta_1), \ldots,(Y_n,\theta_n)$ are sampled independently from the joint $(Y,\theta)$ distribution in \eqref{eq:hierarchy}, but only the $Y$'s are observed.  DNA microarray technologies and the related statistical problems spurred much of the growth in this area; see \citet{efron2008, efron2010book}.  This model has two key features:
\begin{itemize}
\item the number of cases $n$ is typically very large, say tens of thousands; 
\vspace{-2mm}
\item the cases are parallel in the sense that each $Y_i$ has a corresponding (latent) $\theta_i$ which is an independent copy of the single $\theta$ seen in the previous sections.  
\end{itemize}
Together, these two features provide the following intuition: by treating the observed data $(Y_1,\ldots,Y_n)$ as a proxy for the unobserved parameters $(\theta_1,\ldots,\theta_n)$, a large independent sample from $F$, it should be possible to estimate $F$ empirically.  

The canonical high-dimensional model is the normal mean model, i.e., $Y_i \mid \theta_i \sim {\sf N}(\theta_i, 1)$, $i=1,\ldots,n$.  This seemingly simple model has given rise to many fundamental developments in statistics.  Indeed, \citet{stein1981} showed that the high-dimensionality alone is cause for statisticians to rethink their approach.  The fundamental idea behind modern approaches to high-dimensional problems is that inference can be improved by sharing information between cases, and frequentists and Bayesians alike have incorporated this idea into their respective analyses; e.g., FDR controlling procedures \citep{benjamini.hochberg.1995} and hierarchical Bayes methods \citep{scottberger}.  The empirical Bayes approach \citep[e.g.,][]{robbins1964} falls somewhere in between the frequentist and Bayesian extremes.  It starts with a Bayesian model and uses the observed data $Y_1,\ldots,Y_n$ to estimate the prior.  This easily and naturally facilitates the sharing of information between cases.  Parametric empirical Bayes methods have received considerable attention; see \citet{efron2010book} and the references therein.  The James--Stein estimator is a classical example, where $(\theta_1,\ldots,\theta_n)$ is assigned a Gaussian prior with variance estimated from the data.  But the very-high-dimension of modern problems suggests that the more robust \emph{nonparametric} empirical Bayes methods might be successful.

\subsection{Robbins' nonparametric empirical Bayes}
\label{SS:npeb}

In the high-dimensional case, with $n$ large, it may not be necessary to impose parametric constraints on the unknown prior.  \citet{robbins1964} considered nonparametric estimation of the prior $F$ based on $Y_1,\ldots,Y_n$.  With an estimate $\widehat F_n$ of $F$, the Bayes rule $\delta_F$ can be replaced by a plug-in estimate $\hat\delta_n = \delta_{\widehat F_n}$ to be used in a future decision problem.  

\begin{defn}
\label{def:risk}
Let $\widehat F_n$ be an estimate of $F$ based on data $Y_1,\ldots,Y_n$ from the model \eqref{eq:hierarchy}.  Define $\hat \delta_n = \delta_{\widehat F_n}$ to be the decision rule obtained by plugging in $\widehat F_n$ for the true $F$ in the Bayes rule $\delta_F$.  Then $\rho_n(F) = \rho(\hat\delta_n, F)$ represents the risk incurred by using $\hat \delta_n$ in a future decision problem.  
\end{defn}

The decision rule $\hat \delta_n$ in Definition~\ref{def:risk} is called an empirical Bayes rule and $\rho_n(F)$ the corresponding empirical Bayes risk.  It is important to note that Definition~\ref{def:risk} is not the same as that of \citet{robbins1964} and others; this classical risk involves an expectation over the observed data sequence.  Therefore, Robbins' empirical Bayes risk is a \emph{number} whereas $\rho_n(F)$ is a \emph{random variable}.  

\begin{remark}
The decision-theoretic formulation of the empirical Bayes problem given above is based on minimizing the risk in a future decision problem.  In practice, however, we are often interested in the ``compound problem'' of making decisions about $\theta_1,\ldots,\theta_n$ simultaneously.  The relationship between an empirical Bayes problem and the so-called compound decision problem is discussed in \citet{samuel1967} and \citet{copas1969}.  The Bayes rule for the compound problem is to apply the Bayes rule for a future decision to each component problem.  Therefore, the natural approach that is typically used in high-dimensional applications is to apply the resulting empirical Bayes rule for a future decision to each component problem.  See Section~\ref{S:examples}.  
\end{remark}

\section{Asymptotic optimality}
\label{S:optimality}

By definition, $\rho_n(F) \geq \rho(F)$ almost surely.  But, as $n \to \infty$, we have more data with which to construct $\widehat F_n$, so we might expect to be able to get close to the Bayes risk asymptotically.  It is in this regard that we measure the performance of $\hat \delta_n$.  

\begin{defn}
\label{def:optimality}
Let $\FF$ be a given collection of probability measures, assumed to contain the true prior $F$.  A sequence of decision functions $\hat \delta_n$ is asymptotically optimal relative to $\FF$ if $\rho_n(F) \to \rho(F)$ almost surely for all $F \in \FF$.
\end{defn}

Asymptotic optimality in Defintion~\ref{def:optimality} is different than that of Robbins.  Indeed, Robbins' asymptotic ``$\E$-optimality'' includes an additional expectation over the data sequence $Y_1,Y_2,\ldots$.  While asymptotic optimality need not imply asymptotic $\E$-optimality, the difference is important from a statistical point of view: the former means that, for large $n$, the decision procedure has small risk for (almost) \emph{every data sequence}, whereas the latter means the decision procedure does well only \emph{on average}.  Clearly, asymptotic $\E$-optimality means very little to a Bayesian who does not believe in averaging over $\YY$.  

Next is a general theorem on asymptotic optimality, similar to that for asymptotic $\E$-optimality found in \citet{deelyzimmer1976}.

\begin{thm}
\label{thm:optimality}
For $F \in \FF$, assume that $\hat \delta_n(y) \to \delta_F(y)$ almost surely for $\mu$-almost all $y$, that $L(\hat \delta_n(y),\theta) \to L(\delta_F(y), \theta)$ almost surely for $(\mu \times F)$-almost all $(y,\theta)$, and that there exists a sequence of integrable functions $h_n(y,\theta) = h_n(y,\theta; Y_1,\ldots,Y_n)$ such that 
\begin{itemize}
\item $h_n(y,\theta) \to h(y,\theta)$ almost surely for $(\mu \times F)$-almost all $(y,\theta)$,
\vspace{-2mm}
\item $L(\hat \delta_n(y),\theta) \leq h_n(y,\theta)$ almost surely for all $n$ and for $(\mu \times F)$-almost all $(y,\theta)$, and
\vspace{-2mm}
\item $\int_\YY \int_\Theta h_n(y,\theta) p_\theta(y) \,dF(\theta)\,d\mu(y) \to \int_\YY \int_\Theta h(y,\theta) p_\theta(y)\,dF(\theta)\,d\mu(y) < \infty$ almost surely.  
\end{itemize}
Then $\hat \delta_n$ is asymptotically optimal relative to $\FF$.    
\end{thm}

\begin{proof}[\indent Proof]
The proof is a simple application of the dominated convergence theorem or, more specifically, the main theorem of \citet{pratt1960}.  Write 
\begin{equation}
\label{eq:interchange}
\lim_{n \to \infty} \rho_n(F) = \lim_{n \to \infty} \int_\YY \int_\Theta L(\hat \delta_n(y),\theta) p_\theta(y) \,dF(\theta) \,d\mu(y). 
\end{equation}
It remains to show that, with probability 1, limit and integration can be interchanged.  Let $\A^\infty$ be the appropriate $\sigma$-algebra on $\YY^\infty$ and let $\prob_F^\infty$ be the distribution measure of $Y_1,Y_2,\ldots$.  There are five ``$\prob_F^\infty$-almost surely''~assumptions made in the theorem: one about $\hat \delta_n$, one about the loss $L(\hat \delta_n,\theta)$, and three about $h_n$.  Let $A_1,\ldots,A_5 \in \A^\infty$ denote the events where these respective assumptions are true.  By assumption, $\prob_F^\infty(A_i) = 1$, for $i=1,\ldots,5$.  For any data sequence in $A_1 \cap \cdots \cap A_5$, interchange of limit and integration in \eqref{eq:interchange} holds by Pratt's theorem.  The claim follows since $\prob_F^\infty(A_1 \cap \cdots \cap A_5) = 1$.    
\end{proof}

The assumption that the loss converges can be easily checked in practice.  For example, to estimate a real $\theta$, the loss $L(a,\theta)$ is typically a continuous function of the action (estimate) $a$ and the parameter $\theta$ such as $L(a,\theta) = (a-\theta)^2$.  In other problems, such as hypothesis testing, the action space $\AA$ has only a finitely many elements and the desired loss convergence obtains in all but the strangest of cases.  


\section{Nonparametric estimation of the prior $F$}
\label{S:estimation}

\subsection{Predictive recursion}
\label{SS:pr}

Robbins' nonparametric empirical Bayes analysis requires a nonparametric estimate of the prior $F$.  There are a variety of methods available for this task, e.g., nonparametric maximum likelihood, deconvolution, etc.  Here I focus on a relatively new method, namely \emph{predictive recursion}.  It is interesting that the predictive recursion (PR) algorithm boils down to a stochastic approximation \citep{martinghosh}, one of Robbins' other great contributions \citep[see][]{robbinsmonro, lai}. 

PR is a fast, stochastic algorithm for estimating mixing distributions in nonparametric mixture models.  PR's original motivation was as a computationally efficient alternative to Markov chain Monte Carlo methods in fitting Bayesian Dirichlet process mixture models \citep{nqz, newton02}.  If, or to what extent, the PR estimates approximate the Bayesian estimates in a Dirichlet process mixture model remains an open question; however, simulations and theoretical arguments in \citet{tmg} indicate that PR is indeed an attractive alternative.  

Let $\prob_F$ be the marginal distribution of the individual $Y_i$'s, having density $p_F(y) = \int p_\theta(y) \,dF(\theta)$ with respect to $\mu$.  For observations $Y_1,\ldots,Y_n$ from $\prob_F$, the PR algorithm for nonparametric estimation of $F$ and $p_F$ is as follows.

\begin{PRalg}
Choose a starting value $F_0$ to initialize the algorithm, and a sequence of weights $\{w_i: i \geq 1\} \subset (0,1)$.  For $i=1,\ldots,n$, repeat 
\begin{align}
p_{i-1}(y) & = \int p_\theta(y) \,dF_i(\theta), \label{eq:pr.mixture} \\
dF_i(\theta) & = (1-w_i) \,dF_{i-1}(\theta) + w_i \, p_\theta(Y_i) \,dF_{i-1}(\theta) / p_{i-1}(Y_i). \label{eq:pr.mixing} 
\end{align}
Produce $F_n$ and $p_n=p_{F_n}$ as the final estimates of $F$ and $p_F$, respectively.  
\end{PRalg}

An important property of PR is its speed and ease of implementation.  Also, PR has the unique ability to estimate a mixing distribution $F$ which is absolutely continuous with respect to any user-specified dominating $\sigma$-finite measure $\nu$ on $\Theta$.  Indeed, it is easy to see that $F_n$ dominated by $F_0$ for all $n$.  Therefore, if $F_0$ has a density with respect to $\nu$, then so does $F_n$.  Compare this to the nonparametric maximum likelihood estimate which is a.s.~discrete \citep{lindsay1995}.  This property is particularly important in the multiple testing application in \citet{mt-test}, as identifiability of the model parameters requires a careful handling of the underlying dominating measure.  

I should also point out that the PR estimates $F_n$ and $p_n$ depend on the order in which the data enter the algorithm.  This dependence is typically weak, especially for large $n$, but to remove this dependence, it is standard to average the PR estimates over several randomly chosen data permutations; see Section~\ref{S:examples}.  \citet{tmg} make a formal case, based on Rao--Blackwellization, for averaging PR over permutations.

A summary of PR's convergence properties was given in Section~\ref{S:intro}.  Here I state a theorem of \citet{mt-rate} which describes the behavior of $F_n$ and $p_n$ in the case where $\Theta$ is not necessarily finite.  This result will be used to establish asymptotic optimality of PR-based nonparametric empirical Bayes rules in Section~\ref{SS:pr.npeb}.  Let $\FF$ be (a subset of) the set of probabilities measures $F$ on $\Theta$.  For densities $p$ and $p'$ on $\YY$, let $K(p,p') = \int \log(p/p') p\,d\mu$ be the Kullback--Leibler divergence of $p'$ from $p$.  Consider the following set of assumptions. 

\begin{itemize}
\item[A1.] The set $\FF$ of candidate $F$'s is precompact in the weak topology.
\vspace{-2mm}
\item[A2.] $\theta \mapsto p_\theta(y)$ is bounded and continuous for $\mu$-almost all $y$.
\vspace{-2mm}
\item[A3.] The PR weights $(w_n) \subset (0,1)^\infty$ satisfy $\sum_n w_n = \infty$ and $\sum_n w_n^2 < \infty$.  
\vspace{-2mm}
\item[A4.] There exists $C < \infty$ such that $\sup_{\theta_1,\theta_2,\theta_3} \int (p_{\theta_1}/p_{\theta_2})^2 p_{\theta_3} \,d\mu \leq C$.
\vspace{-2mm}
\item[A5.] Identifiability: If $p_F = p_{F'}$ $\mu$-almost everywhere for some $F,F' \in \FF$, then $F=F'$.
\vspace{-2mm}
\item[A6.] For any $\eps > 0$ and any compact $\YY' \subset \YY$, there exists a compact $\Theta' \subset \Theta$ such that $\int_{\YY'} p_\theta(y) \,d\mu(y) < \eps$ for all $\theta \in \Theta'$.  
\end{itemize}

\begin{thm}[\citealt{mt-rate}]
\label{thm:pr}
Under A1--A4, $K(p_F, p_n) \to 0$ $\prob_F$-almost surely.  Furthermore, under A1--A6, $F_n \to F$ in the weak topology, $\prob_F$-almost surely.  
\end{thm}

\begin{remark}
\label{re:pr.conditions}
\citet{mt-rate} discuss the conditions and ways they can be relaxed.  Condition A4 is the strongest, but it holds generally for exponential families whose sufficient statistic has bounded moment-generating function on $\Theta$.  
\end{remark}

\begin{remark}
\label{re:pr.weight}
The PR weights are often taken as $w_n = (n+1)^{-\gamma}$ for some $\gamma \in (1/2,1]$, which clearly satisfies A3.  If $\gamma \in (2/3,1]$, then \citet{mt-rate} establish a $o(n^{1-\gamma})$ bound on the Kullback--Leibler rate of convergence.  
\end{remark}

A generalization of the nonparametric mixture model $Y_1,\ldots,Y_n \iid p_F(y)$ is the semiparametric problem where, in addition to the unknown prior/mixing distribution $F$, there is a finite-dimensional parameter $\omega$ to be estimated as well.  \citet{mt-prml} propose an extension of the PR algorithm for simultaneous estimation of $(F,\omega)$, based on the interesting construction of a PR-based likelihood function for $\omega$.  They show that this PR-based likelihood function approximates the marginal likelihood under a Bayesian Dirichlet process mixture model.  Applications of this methodology can be found in \citet{mt-test} and \citet{prml-finite}.  

\subsection{Nonparametric empirical Bayes via PR}
\label{SS:pr.npeb}

The advantage of Robbins' brand of nonparametric empirical Bayes is that, once an estimate $F_n$ of $F$ is available, the inference problem is straightforward.  That is, one simply finds the Bayes rule $\delta_F$ that depends on the unknown $F$, and then replaces that with $\delta_n = \delta_{F_n}$.  PR seems to be ideally suited to this problem.  The available asymptotic theory for the PR estimate $F_n$ in Theorem~\ref{thm:pr} will be applied, along with Theorem~\ref{thm:optimality}, to prove that the PR-based plug-in nonparametric empirical Bayes rule is asymptotically optimal in the sense of Definition~\ref{def:optimality}.  

Start with the hypothesis testing problem described in Section~\ref{SS:bayes}.  If $F_n$ and $p_n$ are estimates of $F$ and $p_F$ based on the PR algorithm, then the corresponding empirical Bayes rule $\delta_n(y) = \delta_{F_n}(y)$ is given by
\begin{equation}
\label{eq:PRtest}
\delta_n(y) = \begin{cases} a_0 & \text{if $F_n(\Theta_0 \mid y) > r$} \\ a_1 & \text{if $F_n(\Theta_0 \mid y) \leq r$} \end{cases} 
\end{equation}
We now prove the following asymptotic optimality result.

\begin{prop}
\label{prop:test-optim}
If $\prob_\theta$ is a continuous distribution, $L(a,\theta)$ is as described in Section~\ref{SS:bayes}, and the conditions of Theorem~\ref{thm:pr} hold, then $\delta_n$ in \eqref{eq:PRtest} is asymptotically optimal with respect to $\FF$ in the sense of Definition~\ref{def:optimality}.  
\end{prop}

\begin{proof}[\indent Proof]
Under the conditions of Theorem~\ref{thm:pr}, it is clear that $F_n(\Theta_0 \mid y) \to F(\Theta_0 \mid y)$ almost surely for all $y$.  The continuity assumption implies the true posterior probability $F(\Theta_0 \mid y)$ is off the threshold $r$ with probability~1, so it then follows that $\delta_n(y) \to \delta_F(y)$ almost surely for each $y$.  Since the loss $L(a,\theta)$ is uniformly bounded, the choice $h_n(y,\theta) \equiv \sup_{a,\theta} L(a,\theta)$ in Theorem~\ref{thm:optimality} shows that $\delta_n$ is asymptotically optimal.  
\end{proof}

Things are a bit more challenging in the estimation problem in that more conditions are required to establish asymptotic optimality of the PR-based empirical Bayes rule.  Suppose, for example, that $\Theta \subseteq \RR$ and $L(a,\theta) = (a-\theta)^2$, square-error loss.  Then the Bayes rule is the posterior mean and, hence, the PR-based empirical Bayes rule is
\[ \delta_n(y) = \int_\Theta \theta \,dF_n(\theta \mid y) = \frac{1}{p_n(y)} \int_\Theta \theta p_\theta(y) \,dF_n(\theta). \]
Notice that conditions of Theorem~\ref{thm:pr} are not enough to conclude $\delta_n(y) \to \delta_F(y)$ a.s.~for each $y$.  For this to follow, we need $\theta \mapsto \theta p_\theta(y)$ to be bounded for each $y$; this is satisfied if, for example, $p_\theta$ is a ${\sf N}(\theta,1)$ density.  Since the loss is unbounded in general, finding a bounding sequence $h_n(y,\theta)$ as in Theorem~\ref{thm:optimality} must be done carefully case-by-case.  However, a general optimality result holds under the extra condition that $\Theta$ and, hence, $\AA$ are compact.  This is not really a restriction, in this case, since verifying condition A1 in Theorem~\ref{thm:pr} usually requires $\Theta$ to be compact anyway.  

\begin{prop}
\label{prop:est-optim}
If $L(a,\theta)$ is bounded on $\AA \times \Theta$, $\theta \mapsto \theta p_\theta(y)$ is bounded for each $y$, and the conditions of Theorem~\ref{thm:pr} hold, then the PR-based empirical Bayes rule $\delta_n(y)$ is asymptotically optimal in the sense of Definition~\ref{def:optimality}.  
\end{prop}

\begin{proof}[\indent Proof]
Take $h_n(x,\theta) \equiv \sup_{a,\theta} L(a,\theta)$ and apply Theorem~\ref{thm:optimality}.
\end{proof}

\section{Baseball example}
\label{S:examples}


\subsection{Model, data, and objectives}

Empirical Bayes analysis of hitting performance in major league baseball has been a recurring theme in the literature, e.g., \citet{efronmorris1973, efronmorris1975}, \citet{brown2008}, \citet{muralidharan2009}, and \citet{jiang.zhang.2010}.   In these papers, focus has been on using data on each players' batting performance in the first half of the season to simultaneously predict their batting performance in the second half of the season.  Due to the large number of players in consideration (roughly 500 in the analysis that follows), prediction is improved by pooling information across the different players.  Empirical Bayes is a particularly convenient way to perform this information sharing.  

The model setup is as follows.  In the first half of the season, Player $i$, $i=1,\ldots,n$, has $n_i$ at-bats, and his number of hits $Y_i$ is modeled as $Y_i \sim {\sf Bin}(n_i,\theta_i)$, where $\theta_i$ represents Player $i$'s latent hitting ability.  This is an unrealistic setup (for a variety of reasons), but makes for a relatively simple analysis.  The goal is first to estimate $(\theta_1,\ldots,\theta_n)$ based on data for all $n$ players from the first half of the season.  Then these estimates will be used to generate a prediction for the second half hitting performance, and the performance of the estimation procedure will be judged by how well the method predicts.  

The data consists of batting records for each player involved in the 2005 major league baseball season.  In Brown's study, he splits the data into first and second half statistics---these are the ``training'' and ``testing'' portions.  Some players had insufficient number of at-bats, and were removed from the sample.  So the essentially both training and testing portions contain data for the same players; the only caveat is that a few players with sufficient number of at-bats in the first half but an insufficient number in the second half (perhaps due to injury).  Brown also introduces a suitable variance-stabilizing transformation to take the original binomial data to approximately normal data, so that the standard procedures (e.g., James--Stein) can be easily applied.  Specifically, the new model is $X_i \sim {\sf N}(\xi_i,1/4n_i)$ (approximately), for $i=1,\ldots,n$, where $\xi_i = \arcsin\sqrt{\theta_i}$, and the goal is to simultaneously estimate $(\xi_1,\ldots,\xi_n)$ based on the first half data and then give a prediction of the observed $(X_1',\ldots,X_n')$ in the second half.  The reader is referred to \citet{brown2008} for details on the variance-stabilizing transformation [equation (2.2) in \citet[][p.~121]{brown2008}] and prediction error calculations [expression $\widehat{\text{\sc tse}}$ in \citet[][p.~126]{brown2008}]; suffice it to say that small prediction error is preferred.

\subsection{Results}

For data $X_i \sim {\sf N}(\xi_i,1/4n_i)$, a variety of methods are available for estimating $(\xi_1,\ldots,\xi_n)$.  One is to estimate $\xi_i$ with $X_i$; the performance of this ``naive'' procedure is taken as the baseline for comparison.  Another option is to estimate all $\xi_i$'s with the common sample mean $\overline{X}$, the group mean.  \citet{brown2008} describes a number of other, more sophisticated Bayes and empirical Bayes methods.  

\citet{muralidharan2009} describes a method---called mixfdr---which is based on a finite mixture model for the unknown prior distribution.  This method can be naturally applied directly to the binomial data, the $(Y_1,\ldots,Y_n)$, so the transformed data is not needed.  In this setting, he models the unknown prior $f(\theta)$ as a finite mixture of beta densities, and uses Type~II maximum likelihood to estimate the mixture model parameters.    

PR can also be applied to the binomial data directly.  The conditions for Theorem~\ref{thm:pr} can readily checked for this binomial problem; see Remark~\ref{re:pr.conditions}.  For the initial guess $f_0$, I employ some basic knowledge about the context to make an informative choice.  In particular, for pitchers, who tend to have lower batting averages, I take $f_0$ to be a ${\sf Beta}(30, 120)$ distribution, so that the mean is at 0.200.  Likewise, for non-pitchers, who typically have higher batting averages for pitchers, I take $f_0$ to be a ${\sf Beta}(30, 90)$, so that the mean is at 0.250.  For the weight sequence, consider $w_i = (i+1)^{-\gamma}$ as in Remark~\ref{re:pr.weight}.  If $\gamma$ is treated like a tuning parameter, we can choose the value of $\gamma$ to minimize Brown's prediction error.  This optimization problem was solved for the pitcher and non-pitcher sets separately, giving $\gamma=0.5$ for pitchers and $\gamma=0.9$ for non-pitchers.  Lastly, the results of the PR algorithm are averaged over 100 random permutations of the data to remove dependence on the original ordering.  

In Table~\ref{tab:baseball}, I repeat a portion of Muralidharan's Table~1, together with the corresponding PR results.  The results in the top portion of the table are based on the transformed data.  Since both PR and mixfdr are applied to the original data, the reported predictions used are the posterior means of $\arcsin\sqrt{\theta_i}$ based on the estimated priors.  In this case, the PR method is a clear winner when applied to the pitcher portion of the data, and is competitive in the non-pitcher portion, beating all methods except mixfdr, including the theoretically strong nonparametric empirical Bayes procedure of \citet{brown.greenshtein.2009}.  That PR performs well in the smaller-scale pitcher portion of the data suggests that it makes more efficient use of the limited information compared to other methods.  Figure~\ref{fig:baseball} shows both the PR and mixfdr estimates of the prior density $f(\theta)$ for both pitcher and non-pitcher batting averages.  In both cases, I would argue that the PR estimates are much more realistic than the mixfdr estimates.  For pitchers, the mixfdr estimate has some peculiar features.  That there seems to be two subgroups is itself not a concern, but the relative proportions are questionable: among pitchers, there may be a relatively small subgroup who are strong hitters, but the plot indicates that a majority of pitchers fall in this ``extraordinary'' group.  The PR estimate, on the other hand, is smooth and unimodal, with a slight skew to the right indicating a few skillful hitters as outliers in this group of pitchers.  For the non-pitchers, the support of the mixfdr estimate is questionable.  Many major league players hit higher than 0.300 on a regular basis, e.g., Ichiro Suzuki, arguably one of the best hitters in baseball, has a career batting average of 0.324, marked by a $\triangle$ in Figure~\ref{fig:baseball}(b).  This value is an extreme outlier under the mixfdr estimate, but sits nicely at the tip of the upper tail of the PR estimate.  On the other end, there are players who consistently hit near 0.200.  Typically these players are strong at defense to make up for their relatively poor offensive performance.  Henry Blanco, whose career batting average is 0.228, also marked by a $\triangle$ in Figure~\ref{fig:baseball}(b), is one such player.  Overall, this example suggests that PR-driven nonparametric empirical Bayes gives good results in the prediction problem, compared to a variety of methods in both pitcher and non-pitcher portions, and can also give a very reasonable picture of the distribution of latent hitting abilities.

\begin{table}[t]
\begin{center}
\begin{tabular}{ccc}
\hline
& Pitchers & Non-pitchers \\
\hline
\emph{Number of training players} & \emph{81} & \emph{486} \\
\emph{Number of test players} & \emph{64} & \emph{435} \\
Naive & 1 & 1 \\
Group mean & 0.127 & 0.378 \\
Parametric EB (MM) & 0.129 & 0.387 \\
Parametric EB (ML) & 0.117 & 0.398 \\
Nonparametric EB & 0.212 & 0.372 \\
James--Stein & 0.164 & 0.359 \\
Hierarchical Bayes & 0.128 & 0.391 \\
\hline 
mixfdr EB & 0.156 & 0.314 \\
{\bf PR-based EB} & {\bf 0.096} & {\bf 0.353} \\
\hline
\end{tabular}
\caption{\label{tab:baseball} Relative prediction errors for various empirical Bayes estimation methods in the baseball data example of \citet{brown2008} and \citet{muralidharan2009}.}
\end{center}
\end{table}

\begin{figure}[t]
\begin{center}
\subfigure[Estimated priors: pitchers]{\scalebox{0.6}{\includegraphics{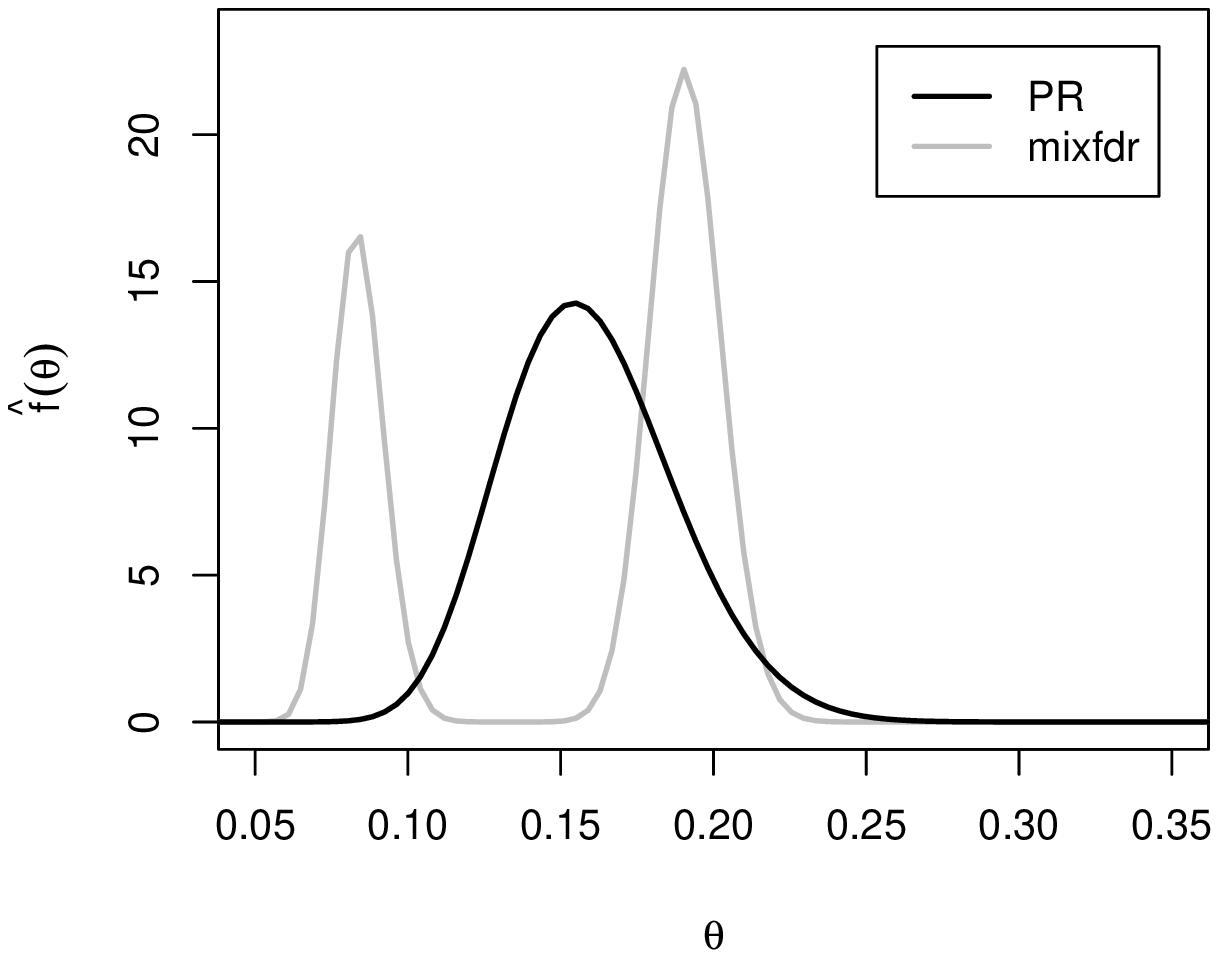}}}
\subfigure[Estimated priors: non-pitchers]{\scalebox{0.6}{\includegraphics{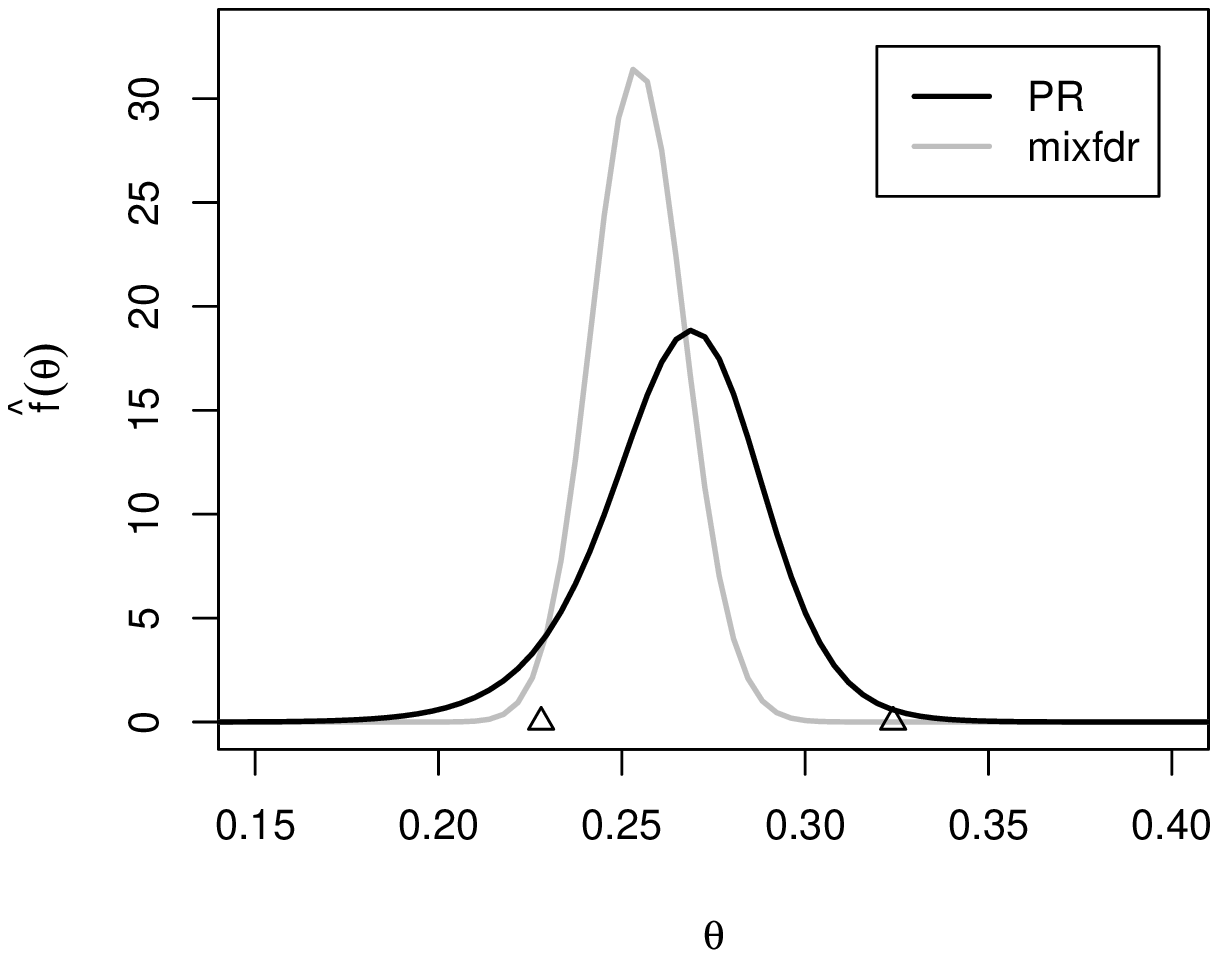}}}
\caption{Plots of estimates of the prior $f(\theta)$ based on PR and Muralidharan's mixfdr.  In panel (b), $\triangle$s mark the career batting averages of Henry Blanco (0.228) and Ichiro Suzuki (0.324), respectively (as of 2012).}
\label{fig:baseball}
\end{center}
\end{figure}

One possible extension of the above analysis is to effectively combine the pitcher and non-pitcher data together to achieve further sharing of information.  Ignoring the information contained in the pitcher/non-pitcher label is not an effective approach.  One possible alternative is to add another parameter to deal with the pitcher/non-pitcher information.  For example, if $X_i=1$ if player $i$ is a pitcher and $X_i=0$ otherwise, then the model could be modified as follows: $Y_i | (X_i,\theta_i) \sim {\sf Bin}(n_i, \omega^{X_i}\theta_i)$, $i=1,\ldots,n$, where $\omega \in (0,1)$ is an unknown shrinkage factor describing the overall discount in pitchers' hitting ability compared to non-pitchers'.  This approach can easily be handled within the predictive recursion marginal likelihood framework \citep{mt-prml}, but I shall omit these details here since it takes us outside the context where PR optimality results are available.

\section{Discussion}
\label{S:discuss}

In this paper I have considered the empirical Bayes approach to statistical inference and its implementation via the PR algorithm.  In particular, I have shown that PR-based empirical Bayes rules are asymptotically optimal under certain conditions.  The question of asymptotic optimality of empirical Bayes estimation in high-dimensional problems where, e.g., the level of sparsity depends on the dimension, is more challenging, and more work is needed to establish this for the PR procedure presented herein.  However, the fact that PR empirically outperforms methods (e.g., the nonparametric empirical Bayes procedure of \citet{brown.greenshtein.2009} appearing in the baseball example above) which are known to be asymptotically optimal in this strong sense suggests that the PR procedure has similar theoretical properties.  

Classical results on empirical Bayes analysis rely heavily on the concept of asymptotic $\E$-optimality.  I argue that asymptotic optimality in Definition~\ref{def:optimality} is more meaningful from a statistical point of view.  In either case, asymptotic optimality is clearly a desirable property; but one could certainly argue that asymptotic optimality is not the only quality to consider.  Robbins and others proposed empirical Bayes rules which were derived from, or at least motivated by, asymptotic optimality considerations.  These procedures often came under criticism since the justification based on asymptotic optimality was not convincing and their performance in real applications was unsatisfactory.  In this era of high-dimensional problems, the sample sizes needed for asymptotic optimality to be meaningful in practice are now readily available.  I argue that a procedure which is both asymptotic optimal and can be justified by other means is most convincing, and here I have shown that PR is such a procedure.  But when $n$ is large, there are many other justifiable alternatives---such as estimating $F$ by the method of maximum likelihood or the method of \citet{deelykruse1968}---which would also lead to asymptotically optimal procedures, so what makes PR stand out?  Although these alternatives have similar asymptotics, in finite samples they typically give estimates of $F$ which are discrete.  This is clearly inappropriate if vague prior information indicates that $F$ is continuous.  Another issue is identifiability.  In the ``two-groups'' problems considered in \citet{mt-test}, $F$ is assumed to have both discrete and continuous components.  The PR algorithm can easily handle this type of vague prior information, whereas maximum likelihood requires additional assumptions, for example, to identify each component.

\bibliographystyle{/Users/rgmartin/Research/TexStuff/asa}
\bibliography{/Users/rgmartin/Research/mybib}

\end{document}